\documentclass[sn-mathphys,Numbered]{sn-jnl}% Math and Physical Sciences Reference Style
\usepackage{graphicx}%
\usepackage{float}%
\usepackage{multirow}%
\usepackage{amsmath,amssymb,amsfonts}%
\usepackage{amsthm}%
\usepackage{mathrsfs}%
\usepackage[title]{appendix}%
\usepackage{xcolor}%
\usepackage{textcomp}%
\usepackage{manyfoot}%
\usepackage{booktabs}%
\usepackage{algorithm}%
\usepackage{algorithmicx}%
\usepackage{algpseudocode}%
\usepackage{listings}%
\usepackage[latin1]{inputenc}
\theoremstyle{thmstyleone}%
\newtheorem{theorem}{Theorem}%  meant for continuous numbers
\newtheorem{proposition}[theorem]{Proposition}%

\theoremstyle{thmstyletwo}%
\newtheorem{example}{Example}%
\newtheorem{remark}{Remark}%
\newtheorem{lemma}{Lemma}%

\theoremstyle{thmstylethree}%

\newcommand{\argmaxR}{\raisebox{-1.8mm}{${{\displaystyle \rm arg\,\! max}
                     \atop {\scriptstyle x \in \mathbb{R}}}$}}

\begin{document}

\title[Article Title]{On a transform of the Vincze-statistic and its exact and asymptotic distribution}


\author{\fnm{Dietmar} \sur{Ferger}}\email{dietmar.ferger@tu-dresden.de}



\affil{\orgdiv{Fakult\"{a}t Mathematik}, \orgname{Technische Universit\"{a}t Dresden}, \orgaddress{\street{Zellescher Weg 12-14}, \city{Dresden}, \postcode{01069}, \country{Germany}}}

%\affil[2]{\orgdiv{Department}, \orgname{Organization}, \orgaddress{\street{Street}, \city{City}, \postcode{10587}, \state{State}, \country{Country}}}

%\affil[3]{\orgdiv{Department}, \orgname{Organization}, \orgaddress{\street{Street}, \city{City}, \postcode{610101}, \state{State}, \country{Country}}}


\abstract{We motivate a new nonparametric test for the one-sided two-sample problem, which is based on a transform $T$
of the Vincze-statistic $(R,D)$. The exact and asymptotic distribution of $T$ is derived. The fundamental idea can also be applied to the
two-sided problem. Here too, a limit theorem can be proved.
}

\keywords{two-sample problem, Vincze-statistic, exact $p$-values, empirical process, Maxwell-Boltzmann distribution}


\pacs[MSC Classification]{62E15,62E20, 62G10.}

\maketitle

\section{Introduction}
Let $X_1,\ldots,X_n$ and $Y_1,\ldots,Y_m$ be two independent samples of i.i.d. real random variables with
continuous distribution functions $F$ and $G$, respectively. For deciding wether the hypothesis $H_0: F=G$ or the one-sided alternative
$H_1: F \ge G, F \neq G$ is true the Smirnov-test uses $D_{nm}:= \sup_{x \in \mathbb{R}} \{F_n(x)-G_m(x)\}$, where $F_n$ and
$G_m$ are the empirical distribution functions of $X_1,\ldots,X_n$ and $Y_1,\ldots,Y_m$, respectively.
The distributional behavior of $D_{nm}$ under the hypothesis $H_0$ is well-known for a long time.
Smirnov \cite{Smirnov} shows in case $n=m$ that
$$
 \mathbb{P}(\sqrt{n/2} D_{nn} \le x) \rightarrow 1-\exp(-2 x^2) \quad \text{ for all } x \ge 0.
$$
Gnedenko and Korolyuk \cite{Gnedenko} find the exact distribution of $D_{nn}$ and later on Korolyuk \cite{Korolyuk} extends the result to
$D_{nm}$, when $m$ is an integral multiple of $n$, $m=n p$ with $p \in \mathbb{N}$.\\

Next, let $Z_1,\ldots,Z_{n+m}$ be the \emph{pooled sample}, i.e. $Z_i:=X_i, 1 \le i \le n,$ and $Z_{n+i}:= Y_i, 1 \le i \le m$.
If $Z_{(1)} \le Z_{(2)} \le \ldots, Z_{(n+m)}$ denote the pertaining order statistics, then
$$
 D_{nm}=\max_{1 \le i \le n+m} \{F_n(Z_{(i)})-G_m(Z_{(i)})\}.
$$
Thus
$$
 R_{nm}:= \min\{1 \le i \le n+m: F_n(Z_{(i)})-G_m(Z_{(i)})=D_{nm}\}
$$
is the smallest index among those order statistics $Z_{(i)}$ at which the two-sample process
$\{F_n(x)-G_m(x): x \in \mathbb{R}\}$ attains its maximal value $D_{nm}$. In the special simple test situation
$H_0^*:F=G=U(0,1)$ against $H_1^a: F=U(0,1), G=U(a,1)$, for a fixed $a \in (0,1)$, Vinzce \cite{Vincze1,Vincze2} uses the pair
$(R_{nm},D_{nm})$ instead of the single $D_{nm}$. More precisely, in the case $m=n$ he determines the likelihood ratio of
$(D_{nn},R_{nn}/n-D_{nn})$ for $H_1^a$ and $H_0^*$ and then uses the corresponding Neyman-Pearson test without randomizing part.
As it turns out the Neyman-Pearson test has a significantly superior power in comparison to the Smirnov-test, confer also Koul and Quine \cite{Koul}.
Under $H_0$ the exact distribution of the so-called \emph{Vincze-statistic} $(R_{nm},D_{nm})$ has been computed by Vincze \cite{Vincze0} for equal sample sizes $(m=n)$ and for general sample sizes by Steck and Simmons \cite{Steck}. Gutjahr \cite{Gutjahr} gives an explicit formula, see Proposition \ref{Gutjahr} below.
Note that $(R_{nm},D_{nm})$ is a discrete random variable with range
\begin{equation} \label{range}
 \{1,\ldots,n+m\} \times \{\frac{k}{m}: 0 \le k \le m\}.
\end{equation}
 Moreover, it is distribution free under $H_0$. Therefore, we can apply Corollary 2.2.7 and Remark 2.2.9 of Gutjahr \cite{Gutjahr}, which yield the following basic result. Here, $[x], x \in \mathbb{R},$ denotes the floor function.\\
%even under the alternative $H_1^a$.

\begin{proposition}\label{Gutjahr}(\textbf{Gutjahr}) Assume that $F=G$ and that $m=np$ with $p \in \mathbb{N}$. Then for every $r \in \{1,\ldots,n+m\}$ and every $k \in \{0,1,\ldots,m\}$ such that
\begin{equation} \label{NB}
 s:= (r+k) \frac{n}{n+m} \in \mathbb{N} \cup \{0\} \; \text{ and } \; s \le \min\{r,n\}
\end{equation}
the following equation holds:
\begin{align*}
 &P_{nm}(r,k):= \mathbb{P}(R_{nm}=r, D_{nm}= \frac{k}{m})\\
 &=\frac{\binom{n+m-r+1}{n-s}}{\binom{n+m}{n}} \frac{k+1}{n+m-r+1} \sum_{j=0}^{[\frac{n}{n+m}k]} \frac{\frac{m}{n}}{r+k-1-j \frac{n+m}{n}} \binom{j \frac{n+m}{n}-k}{j} \binom{r+k-1-j \frac{n+m}{n}}{s-1-j}
\end{align*}
If $(r,k)$ in the range (\ref{range}) does not satisfy (\ref{NB}), then $\mathbb{P}(R_{nm}=r, D_{nm}= \frac{k}{m})=0$.
\end{proposition}

\vspace{0.3cm}
With a view to Vincze's idea, our approach is to use a transform $T_{nm}=h(R_{nm},D_{nm})$ of the Vincze-statistic $(R_{nm},D_{nm})$ as follows:
\[
T_{nm} :=
\begin{cases}
   \frac{D_{nm}}{\sqrt{R_{nm} (n+m-R_{nm})}}     & \text{ if } R_{nm} \in \{1, \ldots, n+m-1\}, \\
  0      & \text{ if } R_{nm}=n+m.
\end{cases}
\]

By Remark \ref{Rnmequalnplusm} the event $\{R_{nm}=n+m\}$ has positive probability $\frac{m}{(n+m)(n+m-1)}$ and therefore cannot be ignored.\\

The paper is structured as follows: In section 2 we motivate the upper $T_{nm}$-test (called \emph{$V$-test}) and determine the exact null-distribution of $T_{nm}$ in case that $m$ is an integral multiple of $n$. In section 3 it is shown that $\sqrt{nm(n+m)} T_{nm}$ converges in distribution to a random variable $Z$, which follows the Maxwell-Boltzmann law. Moreover, we report on the results of a small Monte Carlo simulation study. Here, the $V$-test proves to be superior to the Smirnov-test ($S$-test). In the last section, the basic idea is carried over to the two-sided alternative $H_1^*:F \neq G$.  For the corresponding transformation, we also determine the limit distribution.

\section{Exact null-distribution of $T_{nm}$}
To motivate our test-statistic we investigate $T_{nm}$ under the usual limiting regime, that is $\min\{n,m\} \rightarrow \infty$ and
\begin{equation} \label{limitingregime}
 \frac{n}{n+m} \rightarrow \lambda \in (0,1).
\end{equation}

Let $H_{nm}$ be the empirical distribution function pertaining to the pooled sample $Z_1,\ldots,Z_{n+m}$. Observe that
\begin{equation} \label{Hnm}
 H_{nm} = \frac{n}{n+m} F_n+\frac{m}{n+m} G_m \rightarrow \lambda F+(1-\lambda)G =:H \; \text{almost surely (a.s.)}
\end{equation}
uniformly on $\mathbb{R}$ by the Glivenko-Cantelli Theorem. Assume that $F$ and $G$ in the alternative are such that $D=F-G$ has a unique maximizing point $\tau$ (as for instance when $F$ arises from $G$ as in Example \ref{example} below). According to Lemma \ref{AsymptoticH1} in the appendix $H(\tau) \in (0,1)$  and
\begin{equation} \label{limitTnm}
 (n+m) T_{nm} \rightarrow   \frac{\sup_{x \in \mathbb{R}}\{F(x)-G(x)\}}{\sqrt{H(\tau)(1-H(\tau))}} \; a.s.
\end{equation}
whereas
\begin{equation} \label{limitDnm}
 D_{nm} \rightarrow  \sup_{x \in \mathbb{R}}\{F(x)-G(x)\} \; a.s.
\end{equation}
Notice that the limit in (\ref{limitTnm}) is positive. On the other hand, $(n+m) T_{nm} \rightarrow 0$ in probability under the hypothesis, see Lemma \ref{AsymptoticH0} in the appendix.
Therefore, given a level of significance $\alpha \in (0,1)$ our test rejects the hypothesis $H_0$, if $\sqrt{nm(n+m)} T_{nm}> c_{nm}(\alpha)$, where the
exact critical value $c_{nm}(\alpha)$ can be computed by Theorem \ref{exactdistribution}, where we give an explicit formula for the exact distribution of $T_{nm}$ under $H_0$. If $n$ and $m$ are large the exact critical value can be replaced by the $(1-\alpha)$-quantile of the Maxwell-Boltzmann distribution as a consequence of Theorem \ref{convergencetoMB}, which gives the limit distribution of $\sqrt{nm(n+m)} T_{nm}$.

Notice that the limit in (\ref{limitTnm}) is larger than the limit in (\ref{limitDnm}) by the factor $1/\sqrt{H(\tau)(1-H(\tau))} \ge 2$.
Now, if for instance $F$ differs from $G$ only slightly in the left tails of $G$, then $H(\tau)$ is small, because $H(\tau) \in (G(\tau),F(\tau)).$
Consequently, the factor becomes significantly greater than $2$. As a numerical example assume that $G(\tau) < F(\tau)=0.02$, then the factor is greater than $7$. For that reason we strongly expect that the V-test is much more likely to detect the alternative than the Smirnov-test.\\

\begin{example} \label{example} Let $\tau \in \mathbb{R}$ and $\delta > 1$. Define $F=F_{\tau,\delta}$ by
$$
 F(x):= \left\{ \begin{array}{l@{\quad,\quad}l}
                 \delta G(x) & x \le \tau\\ \beta (G(x)-G(\tau))+ \delta G(\tau) & x > \tau,
               \end{array} \right.
$$
If $\delta G(\tau) < 1$ and $\beta >0$ satisfies (*) $\delta G(\tau)+\beta(1-G(\tau))=1$, then $F$ is a continuous distribution function. Further, one verifies that
$$
 D(x):= F(x)-G(x) = \left\{ \begin{array}{l@{\quad,\quad}l}
                 (\delta-\beta)(1-G(\tau)) G(x) & x \le \tau\\ (\delta-\beta)G(\tau)(1-G(x)) & x > \tau.
               \end{array} \right.
$$
Since $\delta>1$ and therefore by (*) $\beta<1$ we see that $\delta-\beta$ is positive and hence $D=F-G \ge0$.  Moreover, it follows that $\tau$ is a maximizing point of $D$ with $D(\tau)>0$. This shows that $F$ and $G$ lie in the alternative $H_1$. If in addition $G$ is strictly increasing in a neighborhood of $\tau$, then $\tau$ is the unique maximizer.
\end{example}

\vspace{0.3cm}
In the following theorem we use the notation $a\wedge b:= \min\{a,b\}$ for reals $a$ and $b$.\\

\begin{theorem} \label{exactdistribution} Assume that $F=G$ and that $m=np$ with $p \in \mathbb{N}$. Then for each $x \ge 0$ the distribution function $J_{nm}$ of $T_{nm}$ is given by
\begin{equation} \label{dfT}
J_{nm}(x)=\mathbb{P}(T_{nm} \le x)= \sum_{r=1}^{n+m-1} \sum_{k=0}^{[\kappa x] \wedge m} P_{nm}(r,k)+ \frac{m}{(n+m)(n+m-1)},
\end{equation}
where
$$
 \kappa= m \sqrt{r(n+m-r)}.
$$
For $x<0$ the probability is equal to zero.
\end{theorem}

\begin{proof} Since the set in (\ref{range}) is the range of $(R_{nm},D_{nm})$, a decomposition yields:
\begin{align}
& \mathbb{P}(T_{nm} \le x)= \sum_{r=1}^{n+m} \sum_{k=0}^m \mathbb{P}(R_{nm}=r,D_{nm}= \frac{k}{m},T_{nm} \le x) \nonumber \\
&= \sum_{r=1}^{n+m-1} \sum_{k=0}^m \mathbb{P}\Big(R_{nm}=r,D_{nm}= \frac{k}{m},\frac{\frac{k}{m}}{\sqrt{r(n+m-r)}} \le x\Big) \label{fsummand}\\
& +\sum_{k=0}^m \mathbb{P}\Big(R_{nm}=n+m,D_{nm}= \frac{k}{m},0\le x\Big). \label{ssummand}
\end{align}
Since $\frac{\frac{k}{m}}{\sqrt{r(n+m-r)}} \le x \Leftrightarrow  k \le [m \sqrt{r(n+m-r)}\; x]$, it follows that the sum in (\ref{fsummand}) is equal to the first summand in (\ref{dfT}). As to the sum in (\ref{ssummand}) observe that for $r=n+m$ the second requirement in (\ref{NB}) is equivalent to
$k \le 0$. Consequently, $P_{nm}(n+m,k)=0$ for all $k \in \{1,\ldots,m\}$ and therefore (by $x \ge 0$),
\begin{equation} \label{Rnm0}
 \sum_{k=0}^m \mathbb{P}\Big(R_{nm}=n+m,D_{nm}= \frac{k}{m},0\le x\Big)= \sum_{k=0}^{m} P_{nm}(n+m,k)= P_{nm}(n+m,0).
\end{equation}
The pair $(n+m,0)$ with $s=n$ satisfies condition (\ref{NB}), whence after some simple algebra one obtains that $P_{nm}(n+m,0)= \frac{m}{(n+m)(n+m-1)}$, which is the second summand in (\ref{dfT}). Finally, since $T_{nm} \ge 0$, it follows that $\mathbb{P}(T_{nm} \le x)=0$ for all negative $x$.
\end{proof}

\vspace{0.4cm}
\begin{remark} \label{Rnmequalnplusm}
Notice that the left side in (\ref{Rnm0}) is equal to $\mathbb{P}(R_{nm}=n+m)$. So, the last part of the above proof shows that $\mathbb{P}(R_{nm}=n+m)=\frac{m}{(n+m)(n+m-1)}.$
\end{remark}

\vspace{0.4cm}
\begin{remark} \label{code} The source code in Mathematica \cite{Wolfram} that implements Gutjahr's formula is only a few lines long. According to Theorem \ref{exactdistribution}, the distribution function $J_{nm}$ of $T_{nm}$ can therefore be calculated quickley without any problems.
\end{remark}

\section{Asymptotic null-distribution}
We derive the asymptotic distribution of $T_{nm}$ under the hypothesis $H_0$. It turns out that $\sqrt{nm(n+m)} T_{nm}$ converges to the
\emph{Maxwell-Boltzmann} distribution, which has distribution function
$$
 K(x)=2 \Phi(x)- \sqrt{\frac{2}{\pi}} x \exp\{-\frac{x^2}{2}\}-1, \; x \ge 0,
$$
and $K(x)=0$ for all negative $x$. Here, $\Phi$ is the distribution function of the standard normal law $N(0, 1)$.\\

\begin{theorem} \label{convergencetoMB} Assume that $H_0$ holds. Then under the limiting regime (\ref{limitingregime}),
\begin{equation*} %\label{convergenceofCDF}
 K_{nm}(x):= \mathbb{P}\Big(\sqrt{nm(n+m)} T_{nm} \le x\Big) \rightarrow K(x)
\end{equation*}
for all  $x \ge 0$. For negative $x$ the
probabilities $K_{nm}(x), n,m \in \mathbb{N},$ are equal to zero.
\end{theorem}

\begin{proof} Recall that $(R_{nm},D_{nm})$ is distribution free under $H_0$, whence we can assume that $F$ and $ G$ both are equal to the distribution function of $U(0,1)$, the uniform distribution on the
unit interval $(0,1)$. So, $F$ is equal to the identity map $Id$ on $[0,1]$. We start with the representation
\begin{equation} \label{rep}
 \sqrt{nm(n+m)}\; T_{nm} :=
\begin{cases}
    \frac{\sqrt{\frac{nm}{n+m}}D_{nm}}{\sqrt{\frac{R_{nm}}{n+m}(1-\frac{R_{nm}}{n+m})}}     & \text{ if } R_{nm} \in \{1, \ldots, n+m-1\}, \\
  0     & \text{ if } R_{nm}=n+m.
\end{cases}
\end{equation}
Note that $D_{nm}= \sup_{x \in [0,1]}\{F_n(x)-G_m(x)\}$ a.s.
Observe that  by (\ref{Hnm})
\begin{equation} \label{HnmH0}
 ||H_{nm}-Id||:=\sup_{x \in [0,1]}|H_{nm}(x)-x| \rightarrow \ \; \text{ a.s.}
\end{equation}
Introduce
$\tau_{nm}:= Z_{(R_{nm})}$.
Since $\frac{R_{nm}}{n+m} = H_{nm}(Z_{(R_{nm})})= H_{nm}(\tau_{nm})$, it follows from (\ref{HnmH0}) that
$$
 \bigg|\frac{R_{nm}}{n+m}-\tau_{nm}\bigg|=|H_{nm}(\tau_{nm})-\tau_{nm}|\le ||H_{nm}-Id|| \rightarrow 0 \; \text{ a.s.}
$$
As a consequence the sequences
$(\sqrt{\frac{nm}{n+m}} D_{nm}, \frac{R_{nm}}{n+m})$ and $(\sqrt{\frac{nm}{n+m}} D_{nm}, \tau_{nm})$ are asymptotical stochastically equivalent.

Our basic idea is to write $(\sqrt{\frac{nm}{n+m}} D_{nm}, \tau_{nm})$ as a functional of the two-sample empirical process
$$
 \alpha_{nm}(x):= \sqrt{\frac{nm}{n+m}}\{F_n(x)-G_m(x)\}, \; x \in [0,1].
$$
To that let $(D[0,1],s)$ be the Shorokhod-space. For every $f \in D[0,1]$ define $M(f):=\sup_{t \in [0,1]} f(t)$ and $A(f):=\{t \in [0,1]: \max\{f(t),f(t-)\}=M(f)\}$ with the convention $f(0-):=f(0)$. By Lemma A.2 in Ferger \cite{Ferger1} $A(f)$ is a non-empty compact subset of [0,1], whence the \emph{argsup-functional} $a(f):=\min A(f)$ is well defined. By definition of
$R_{nm}$ the random variable $\tau_{nm}$ is the smallest maximizing point of $F_n-G_m$, i.e. $\tau_{nm}=a(F_n-G_m)$.
Since $a(c f)=a(f)$ for every positive constant $c$, it follows that
$\tau_{nm} =a(\alpha_{nm})$. Similarly, $M(c f)=c M(f)$ and thus $\sqrt{\frac{nm}{n+m}} D_{nm}= M(\alpha_{nm})$. The functional $L:=(M,a):(D[0,1],s) \rightarrow \mathbb{R}^2$ is Borel-measurable by Lemma A.3 in Ferger \cite{Ferger1} and
continuous on the subset $C_u:=\{f \in C[0,1]: f \text{ has a unique maximizing point}\} \subseteq D[0,1]$. To see this
note that $M$ is continuous on $C[0,1] \supseteq C_u$ and $a$ is continuous on $C_u$ by Lemma A.4 in Ferger \cite{Ferger1}. It is well-known that
$\alpha_{nm} \stackrel{\mathcal{D}}{\rightarrow} B$ in $(D[0,1],s)$, where $B$ is a Brownian bridge.
Now, $B \in C_u$ almost surely, whence an application of the
Continuous Mapping Theorem (CMT) yields that $$(\sqrt{\frac{nm}{n+m}} D_{nm}, \tau_{nm}) = L(\alpha_n) \stackrel{\mathcal{D}}{\rightarrow} L(B)= (M(B),a(B))$$
and by the asymptotical stochastical equivalence we obtain that
$$(\sqrt{\frac{nm}{n+m}} D_{nm}, \frac{R_{nm}}{n+m}) \stackrel{\mathcal{D}}{\rightarrow} (M(B),a(B)).$$
Let $h:\mathbb{R} \times (0,1] \rightarrow \mathbb{R}$ be defined by $h(x,y):=x/\sqrt{y(1-y)}$ for $(x,y) \in \mathbb{R} \times (0,1)=: G$ and $h(x,y):=0$ otherwise.
Obviously $h$ is continuous on $G$. Moreover, $a(B) \in (0,1)$ almost surely. (Indeed, $a(B)$ is uniformly distributed on $(0,1)$.) Thus (\ref{rep}) and another application of the CMT gives
$$
 \sqrt{nm(n+m)} T_{nm} =h(\sqrt{\frac{nm}{n+m}} D_{nm}, \frac{R_{nm}}{n+m}) \stackrel{\mathcal{D}}{\rightarrow} h(M(B),a(B))=:Z.
$$
Now the assertion follows from Theorem 1.1 of \cite{Ferger0}, which says that the distribution of $Z$ is equal to the Maxwell-Boltzmann distribution.
\end{proof}

Figure 1 illustrates the convergence of $K_{nm}$ to $K$.

\begin{figure}[H]
\centering
\includegraphics[scale=0.6]{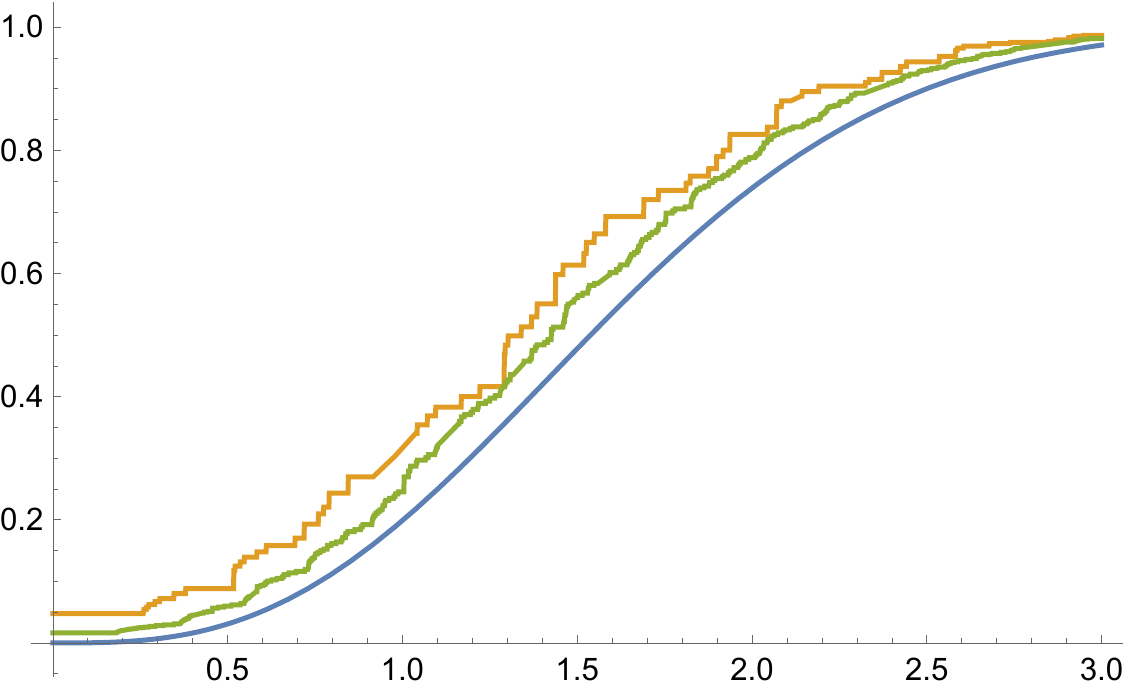}
\caption{Graphs of $K_{10,20}$ (orange), $K_{60,60}$ (green) and $K$ (blue).}
\end{figure}

In a small simulation study with $F(x)=\Phi(x+0.4)$ and $G(x)=\Phi(x)$ (normal location model) and
equal subsample sizes $n=m$ we computed $l=1000$ Monte-Carlo replicates $T_{nn}^{(1)},\ldots,T_{nn}^{(l)}$ and
$D_{nn}^{(1)},\ldots,D_{nn}^{(l)}$ of $T_{nn}$ and $D_{nn}$, respectively. They yielded the corresponding $p$-values $p_V^{(j)}= 1-J_{nn}(T_{nn}^{(j)})$ and
$p_S^{(j)}= L_n(D_{nn}^{(j)}), 1 \le j \le l,$ where
$$
 L_n(x):=\mathbb{P}(D_{nn} > x)=\frac{\binom{2n}{n+[nx]+1}}{\binom{2n}{n}}.
$$

The last equation can be found, for example, in Shorack and Wellner \cite{Shorack} on p.401.
Table 1 below shows the mean $p$-values $p_V=l^{-1} \sum_{j=1}^{l} p_V^{(j)}$ and $p_S=l^{-1} \sum_{j=1}^{l} p_S^{(j)}$.

\begin{table*}[htbp]
%\centering
\caption{mean p-values}
\label{crit}
\begin{tabular}{c|c|c}
\vspace{0.2cm}
$n=m$ & $p_V$ & $p_S$ \\ \hline
 10 & 0.2802 & 0.2521\\
 15 & 0.2542 & 0.2375\\
 20 & 0.2351 & 0.2407 \\
 25 & 0.2086 & 0.2246 \\
 30 & 0.1882 & 0.2124 \\
 35 & 0.1680 & 0.2059 \\
 40 & 0.1498 & 0.1862\\
 50 & 0.1302 & 0.1781\\
 60 & 0.1020 & 0.1540 \\
 70 & 0.0919 & 0.1505\\
 80 & 0.0717 & 0.1218\\
100 & 0.0501 & 0.0954\\
200 & 0.0106 & 0.0351\\
\hline
\end{tabular}
\end{table*}

Note that, with the exception of $n \in \{10,15\}$, the V-test performs significantly better than the Smirnov test

\section{Two-sided alternatives}
Consider the test problem $H_0:F=G$ against the alternative $H_1^*:F \neq G.$ Here, the Kolmogorov-Smirnov test rejects the hypothesis in the case of large values of $D_{nm}^*=\sup_{x \in \mathbb{R}}|F_n(x)-G_m(x)|.$ Since $D_{nm}^*=\max_{1 \le i \le n+m} |F_n(Z_{(i)})-G_m(Z_{(i)})|$, the random variable $R_{nm}^*:= \min\{1 \le i \le n+m: |F_n(Z_{(i)})-G_m(Z_{(i)})|=D_{nm}^*\}$
is the smallest index among those order statistics $Z_{(i)}$ at which the reflected two-sample process
$\{|F_n(x)-G_m(x)|: x \in \mathbb{R}\}$ attains its maximal value. Now, the counterpart $T_{nm}^*$ of $T_{nm}$ is given by:
\[
T_{nm}^* :=
\begin{cases}
   \frac{D_{nm}^*}{\sqrt{R_{nm}^* (n+m-R_{nm}^*)}}     & \text{ if } R_{nm}^* \in \{1, \ldots, n+m-1\}, \\
  0      & \text{ if } R_{nm}^*=n+m.
\end{cases}
\]

\begin{theorem} If $F=G$, then $\mathbb{P}(\sqrt{nm(n+m)} T_{nm}^* \le x) \rightarrow K^*(x)$,
where
\begin{eqnarray*}
K^*(x) &=& 16 \sum_{0\le j<l<\infty}(-1)^{j+l}\frac{\alpha_j\alpha_l}{\alpha_l^2-\alpha_j^2}[\frac{\Phi(\alpha_j x)-1/2}{\alpha_j}-\frac{\Phi(\alpha_l x)-1/2}{\alpha_l}]\\
&+& 4 \sum_{0\le j < \infty} [\frac{\Phi(\alpha_j x)-1/2}{\alpha_j}-x \varphi(\alpha_j x)], \quad x \ge 0.
\end{eqnarray*}
Here, $\alpha_j = 2 j +1, j \in \mathbb{N}_0,$ and $\varphi$ denotes the density of $N(0,1)$.
\end{theorem}

\begin{proof} One shows completely analogously to the proof of Theorem \ref{convergencetoMB} that
$$
 \sqrt{nm(n+m)} T_{nn}^* \stackrel{\mathcal{D}}{\rightarrow} h(M(|B|),a(|B|))=:Z^*.
$$
Thus the assertion follows from Theorem 1.1 of \cite{Ferger0}, which gives the distribution of $Z^*$.
\end{proof}

\section{Appendix}
In this short appendix we describe the asymptotic behavior of $T_{nm}$ under (special) alternatives (Lemma \ref{AsymptoticH1}) and under the hypothesis (Lemma \ref{AsymptoticH0}).\\

\begin{lemma} \label{AsymptoticH1} Let $F$ and $G$ be two different continuous distribution functions such that $F-G \ge 0$ has a unique minimizing point $\tau \in \mathbb{R}$. Then $F(\tau)>G(\tau)$ and $H(\tau) \in (0,1)$. Moreover,
\begin{equation} \label{limitTnmH1}
 (n+m) T_{nm} \rightarrow \frac{\sup_{x \in \mathbb{R}}\{F(x)-G(x)\}}{\sqrt{H(\tau)(1-H(\tau))}} \; a.s.
\end{equation}
\end{lemma}

\begin{proof} First, notice that
\begin{equation} \label{rep2}
 (n+m)\; T_{nm} :=
\begin{cases}
    \frac{D_{nm}}{\sqrt{\frac{R_{nm}}{n+m}(1-\frac{R_{nm}}{n+m})}}     & \text{ if } R_{nm} \in \{1, \ldots, n+m-1\}, \\
  0     & \text{ if } R_{nm}=n+m.
\end{cases}
\end{equation}
Let $\tau_{nm}:= \argmaxR (F_n(x)-G_m(x))$ be the (smallest) maximizing point of $F_n-G_m$. Recall from the proof of Theorem \ref{convergencetoMB} that $\frac{R_{nm}}{n+m}=H_{nm}(\tau_{nm}).$ Thus we obtain
\begin{equation} \label{Rnm}
 \Big|\frac{R_{nm}}{n+m}-H(\tau)\Big| =|H_{nm}(\tau_{nm})-H(\tau)| \le ||H_{nm}-H||+|H(\tau_{nm})-H(\tau)|.
\end{equation}
It follows from Theorem 3.3 in Ferger \cite{Ferger2} that $\tau_{nm} \rightarrow \tau$, whence (\ref{Hnm}), (\ref{Rnm}) and continuity of $H$ yield
that $R_{nm}/({n+m}) \rightarrow H(\tau) \; a.s.$ Conclude with (\ref{limitDnm}) that
\begin{equation} \label{DR}
  (D_{nm},\frac{R_{nm}}{n+m}) \rightarrow (\sup_{x \in \mathbb{R}}\{F(x)-G(x)\}, H(\tau)).
\end{equation}
We know that $F(\tau)-G(\tau) > 0$ is positive, because otherwise it followed that $F=G$ in contradiction to our assumption. Moreover,
$H(\tau) \in (G(\tau),F(\tau))$, because $\lambda \in (0,1).$ Consequently, $H(\tau) \in (0,1)$, whence the limit in (\ref{DR}) lies in
$G=\mathbb{R} \times (0,1).$ Recall the function $h$ in the proof of Theorem \ref{convergencetoMB}, which is continuous on $G$.
By (\ref{rep2}) and (\ref{DR}) it follows that
$$
 (n+m) T_{nm} = h(D_{nm},\frac{R_{nm}}{n+m}) \rightarrow h(\sup_{x \in \mathbb{R}}\{F(x)-G(x)\}, H(\tau)) \; \text{a.s.}
$$
which gives the desired result (\ref{limitTnmH1}).
\end{proof}

\begin{lemma} \label{AsymptoticH0} Assume $H_0$ holds. Then
\begin{equation} \label{limitTnmH0}
 (n+m) T_{nm} \stackrel{\mathbb{P}}{\rightarrow} 0.
\end{equation}
\end{lemma}

\begin{proof} Notice that $(n+m) T_{nm} = \sqrt{\frac{n+m}{nm}} Z_{nm}$, where $Z_{nm} = \sqrt{nm(n+m)} T_{nm} \stackrel{\mathcal{D}}{\rightarrow} Z$
by Theorem \ref{convergencetoMB}. Now, the assertion follows, because the deterministic factor $\sqrt{\frac{n+m}{nm}}$ converges to zero.
\end{proof}

%\vspace{1cm}
%\textbf{Declarations}\\

%\textbf{Compliance with Ethical Standards}: I have read and I understand the provided information.\\

%\textbf{Competing Interests}: The author has no competing interests to declare that are relevant to the content of
%this article.


\end{document}